\documentclass[a4paper,11pt]{article}
\usepackage{amsmath,amsthm,amssymb,amscd,epsfig}

\newtheorem{thm}{Theorem}[section]

\newtheorem{cor}[thm]{Corollary}
\newtheorem{lemma}[thm]{Lemma}
\newtheorem{remark}[thm]{Remark}

\numberwithin{equation}{section}

\def\C{\mathbb C}
\def\N{\mathbb N}
\def\D{\mathbb D}

\def\H{\mathcal H}

\def\diam{\text{diam}}
\def\Lip{\text{Lip}}
\def\ds{\displaystyle}

\title{Sharp Nonremovability Examples for H\"{o}lder continuous quasiregular mappings in the plane}

\author{ \textit{Albert Clop} \and \textit{Ignacio Uriarte-Tuero} 
\thanks{Uriarte-Tuero is a postdoctoral fellow in the Department of Mathematics of the University of Missouri-Columbia. Clop is supported by European Union projects CODY and GALA, and also by projects MTM2007-60062 and 2005-SGR-00774.   \newline \newline AMS (2000) Classification. Primary 30C62, 35J15
 \newline Keywords   Quasiconformal, Hausdorff measure, Removability}}
\date{}

\begin{document}

\maketitle

\begin{abstract}
Let $\alpha\in(0,1)$, $K\geq 1$, and $d=2\frac{1+\alpha K}{1+K}$. Given a compact set $E\subset\C$, it is known that if $\H^d(E)=0$ then $E$ is removable for $\alpha$-H\"older continuous $K$-quasiregular mappings in the plane. The sharpness of the index $d$ is shown with the construction, for any $t>d$, of a set $E$ of Hausdorff dimension $\dim(E)=t$ which is not removable. In this paper, we improve this result and construct compact nonremovable sets $E$ such that $0<\H^d(E)<\infty$. For the proof, we give a precise planar $K$-quasiconformal mapping whose H\"older exponent is strictly bigger than $\frac{1}{K}$, and that exhibits extremal distortion properties.
\end{abstract}

\section{Introduction}

Let $\alpha\in(0,1)$. A function $f:\C\rightarrow\C$ is said to be locally $\alpha$-H\"older continuous, that is, $f\in\Lip_\alpha(\C)$, if
\begin{equation}\label{lipalfa}
|f(z)-f(w)|\leq C\,|z-w|^\alpha
\end{equation}
whenever $z,w\in\C$, $|z-w|<1$. A set $E\subset\C$ is said to be {\it{removable}} for $\alpha$-H\"older continuous analytic functions if every function $f\in\Lip_\alpha(\C)$, holomorphic on $\C\setminus E$, is actually an entire function. It turns out that there is a characterization of these sets $E$ in terms of Hausdorff measures. For $\alpha\in(0,1)$, Dol\v zenko \cite{Do} proved that a set $E$ is removable for $\alpha$-H\"older continuous analytic functions if and only if $\H^{1+\alpha}(E)=0$. When $\alpha=1$, we deal with the class of Lipschitz continuous analytic functions. Although the same characterization holds, a more involved argument, due to Uy \cite{U}, is needed to show that sets of positive area are not removable.\\
\\
The same question may be asked in the more general setting of $K$-quasiregular mappings. Given a domain $\Omega\subset\C$ and $K\geq 1$, one says that a mapping $f:\Omega\rightarrow\C$ is $K$-quasiregular in $\Omega$ if $f$ is a $W^{1,2}_{loc}(\Omega)$ solution of the Beltrami equation,
$$\overline\partial f(z)=\mu(z)\,\partial f(z)$$
for almost every $z\in\Omega$, where $\mu$, the Beltrami coefficient, is a measurable function such that $|\mu(z)|\leq\frac{K-1}{K+1}$ at almost every $z\in\Omega$. If $f$ is a homeomorphism, then $f$ is said to be $K$-quasiconformal. When $\mu=0$, one recovers the classes of analytic functions and conformal mappings on $\Omega$, respectively.\\
\\
We say that $E \subset \C$ is {\it{removable for $\alpha$-H\"older continuous $K$-quasiregular mappings}} if any function $f\in\Lip_\alpha(\C)$, $K$-quasiregular in $\C\setminus E$, is actually $K$-quasiregular on the whole plane. These sets were already studied by Koskela and Martio \cite{KM} and Kilpel\"ainen and Zhong \cite{KZ}, where some sufficient conditions for removability were given in terms of Hausdorff measures and dimension. Later, compact sets $E\subset\C$ satisfying $\H^d(E)=0$, $d=2\frac{1+\alpha K}{1+K}$ were shown to have this property (see \cite{C}). The sharpness of the index $d$ was proved in \cite{C2}. More precisely, given $\alpha\in(0,1)$ and $K\geq 1$, there exists for any $t>d$ a compact set $E$ of dimension~$t$, and a function $f\in\Lip_\alpha(\C)$ which is $K$-quasiregular in $\C\setminus E$, and with no $K$-quasiregular extension to $\C$. In other words, it was shown that there exist nonremovable sets of any dimension exceeding $d$. In \cite{C3}, Problem 3.7 states: Is there some compact set $E$ of dimension $d$, nonremovable for $\alpha$-H\"older continuous $K$-quasiregular mappings? In this paper we construct such a set $E$, which even satisfies $0< \H^d(E) < \infty$. Here we state our result.

\begin{thm}\label{MainTheoremNonRemovableSetsAtCriticalDimension}
Let $\alpha \in (0,1)$ and $K\geq 1$. If $d=2\frac{1+\alpha K}{1+K}$, then there exists a compact set $E \subset \C$ with $0< \H^d(E) < \infty$, nonremovable for $\alpha$-H\"older continuous $K$-quasiregular mappings. 
\end{thm}
\noindent
We want to remark that the above Theorem extends for $K > 1$
the results of Dol\v zenko in \cite{Do} about nonremovable sets for analytic functions in $\Lip_\alpha(\C)$. \\
\\
Let us first have a look at the case $K=1$. Given a compact set $E$ with $\H^{1+\alpha}(E)>0$, by Frostman's Lemma (see for instance \cite[p.112]{M}), there exists a positive Radon measure $\nu$ supported on~$E$, such that $\nu(D(z,r))\leq C\,r^{1+\alpha}$ for any $z\in E$, where $D(z,r)$ is the disk of center $z$ and radius $r$. Thus, the function $h=\frac{1}{\pi z}\ast \nu$ is $\alpha$-H\"older continuous everywhere, holomorphic outside the support of $\nu$ and has no entire extension. From here onwards, $K>1$ unless we specify otherwise. \\
\\
Another similar situation is found in the limiting case $\alpha=0$, in which $\Lip_\alpha(\C)$ should be replaced by~$BMO(\C)$. In this case, a set $E$ is called removable for $BMO$ $K$-quasiregular mappings if every $BMO(\C)$ function $f$, $K$-quasiregular on $\C\setminus E$, is actually $K$-quasiregular on the whole plane. When $K=1$, Kaufman \cite{K} and Kr\'{a}l \cite{Kr} characterized these sets as those with zero length. When $K>1$, it is known (\cite{ACMOU}, \cite{AIM}) that sets with $\H^\frac{2}{K+1}(E)=0$ are removable for $BMO$ $K$-quasiregular mappings. In fact, the appearance of this index $\frac{2}{K+1}$ is not strange. In \cite{A}, Astala showed that for any $K$-quasiconformal mapping $\phi$ and any compact set $E$,
\begin{equation}\label{dimdist}
\frac{1}{K}\left(\frac{1}{\dim(E)}-\frac{1}{2}\right)\leq \frac{1}{\dim(\phi(E))}-\frac{1}{2}\leq K\left(\frac{1}{\dim(E)}-\frac{1}{2}\right).
\end{equation}
Furthermore, both equalities are always attainable, so that if $\dim(E) =t$, then
\begin{equation}\label{MaximumStretchingFormulaForTPrime}
\dim(\phi(E)) \leq t'=\frac{2Kt}{2+(K-1)t}.
\end{equation}
In particular, sets of dimension $\frac{2}{K+1}$ are $K$-quasiconformally mapped to sets of dimension at most $1$, which is the critical point for the analytic $BMO$ situation. Therefore, from equality at \eqref{dimdist}, there exists for any $t>\frac{2}{K+1}$ a compact set $E$ of dimension $t$ and a $K$-quasiconformal mapping $\phi$ that maps $E$ to a compact set $\phi(E)$ with dimension 
$$t'=\frac{2Kt}{2+(K-1)t}>1.$$
In particular $\H^1(\phi(E))>0$. Thus by Frostman's Lemma $\phi(E)$ supports some positive Radon measure $\nu$, having linear growth. Its Cauchy transform $h=\frac{1}{\pi z}\ast\nu$ is a $BMO(\C)$ nonentire function, analytic on $\C\setminus\phi(E)$. Thus, using that $BMO$ is invariant under quasiconformal changes of coordinates \cite{R}, the composition $h\circ\phi$ shows that $E$ is non-removable for $BMO$ $K$-quasiregular mappings.\\
\\ 
Recently, it was shown by Uriarte-Tuero \cite{Ur} that equality at (\ref{dimdist}) may be attained even at the level of measures. More precisely, 
Question 4.2 in \cite{ACMOU} asked whether there exists, for every $K \geq 1$, a compact set $E$ with $0 < \H^\frac{2}{K+1}(E) < \infty$, such that $E\,$� is not removable for some $K$-quasiregular functions in $BMO(\C)$. In \cite{Ur}, the author gives an affirmative answer to this question
by building a highly non-selfsimilar and non-uniformly distributed Cantor-type set $E$ and a $K$-quasiconformal mapping $\phi$ such that 
\begin{equation}\label{uriartesth}
0 < \H^\frac{2}{K+1}(E) < \infty\hspace{1cm}\text{and}\hspace{1cm}0<\H^1(\phi(E))<\infty.
\end{equation}
From the argument above, it then follows that the set $E$ is not removable for $BMO$ $K$-quasiregular mappings, even having positive and finite $\H^\frac{2}{K+1}$ measure. \\
\\
%
%
%
%
Our plan is to repeat the above scheme, but replacing $BMO(\C)$ by $\Lip_\alpha(\C)$. That is, given $d=2\frac{1+\alpha K}{1+K}$, we will construct a compact set $E$ with $0<\H^d(E)<\infty$ and a $\Lip_\alpha(\C)$ function which is $K$-quasiregular on $\C\setminus E$ but not on $\C$.\\
\\
We will start with the construction at \cite{Ur}, to get a compact set $E$ with $0<\H^d(E)<\infty$ and a $K$-quasiconformal mapping $\phi$ such that $0< \H^{d'}(\phi(E))<\infty$, where $d'=\frac{2Kd}{2+(K-1)d}$. Notice that $d'>1$. By Frostman's Lemma, there are nonentire $\Lip_\beta(\C)$ functions with $\beta=d'-1>0$, analytic outside of $\phi(E)$, which in turn induce (by composition) $K$-quasiregular functions on $\C\setminus E$ whose H\"older continuity exponent is, a priori, $\frac{1}{K}\,\beta$, 
because general $K$-quasiconformal mappings belong to $\Lip_{1/K}(\C)$, as Mori's Theorem states. Thus, there is some loss of regularity that might be critical, since
$$\frac{\beta}{K}<\alpha.$$
To avoid these troubles, we will construct in an explicit way the mapping $\phi$. 
This concrete construction allows us to show that $\phi$ exhibits a precise exponent of H\"older continuity given by
\begin{equation}\label{FormulaDOverDPrime}
\frac{d}{d'}=\frac{1}{K}+\frac{K-1}{2K}d
\end{equation}
which is larger than the usual $\frac{1}{K}$. This regularity will be sufficient for our purposes. Notice that since $\dim(E)=d$ and $\dim(\phi(E))=d'$ it is natural to expect $\phi$ to be $\Lip_{d/d'}$. We remark two points in this argument. First, it is precisely the distortion property \eqref{uriartesth} for $\H^d$ and $\H^{d'}$ obtained in \cite{Ur} what allows us to get non removable sets at the critical dimension $d$ (and even with finite $\H^d$ measure.) Second, several technical difficulties will arise when computing the H\"older exponent of $\phi$, because of the fact that the set $E$ is highly nonregular.%
%

\noindent
In terms of notation, $A \lesssim B$ means that there exists a constant $C>0$ such that $A \leq C \; B$. The same letter $C$ in consecutive inequalities may not denote the same constant. $\mid A \mid$ is the area of $A$. If $D = D(z,r)$ is a disk of center $z$ and radius $r$, then $r(D)=r$ also denotes its radius and $\alpha D = D(z,\alpha r)$ for all $\alpha >0$. We say that a measure $\mu$ has growth $t$ if $\mu(D(z,r)) \leq C r^t$ for all $z$. If $t=1$, we say it has linear growth. \\
\\
The paper is structured as follows. In section \ref{BasicConstruction}, we recall from \cite{Ur} how to construct, for any $0<t<2$ and $K>1$, a $K$-quasiconformal mapping $\phi$ and a set $E \subset \C$ such that $0 < \H^t(E) < \infty$ and $0 < \H^{t'}(\phi(E)) < \infty$, $t'=\frac{2Kt}{2+(K-1)t}$. In section \ref{CalculationHolderExponentPhi}, we prove that this $K$-quasiconformal mapping $\phi$ is locally H\"{o}lder continuous with exponent $\frac{t}{t'}$. This section is where most of the new technical difficulties appear. In section \ref{FinalConstructionOfHolderQRMap}, 
we prove Theorem \ref{MainTheoremNonRemovableSetsAtCriticalDimension}.

\section{The basic construction}\label{BasicConstruction}

As we mentioned above, the following theorem is proved in \cite{Ur}:

\begin{thm}\label{TheoremConstructionForUsualHausdorffMeasures}
Let $K>1$. For any $0<t<2$, there exists a compact set $E$ with $0 < \H^t(E) < \infty$ and a $K$-quasiconformal mapping $\phi : \C \rightarrow \C$ such that $0 < \H^{t'}(\phi E) < \infty$, where $t'=\frac{2Kt}{2+(K-1)t}$. 
\end{thm}

For the convenience of the reader, 
we recall from \cite{Ur} the main ideas of the proof.

\begin{proof} (Sketch of proof of Theorem \ref{TheoremConstructionForUsualHausdorffMeasures}.)

We will construct the $K$-quasiconformal mapping $\phi$ as the limit of a sequence $\phi_N$ of $K$-quasiconformal mappings, and $E$ will be a Cantor-type set. To reach the optimal estimates we need to change, at every step in the construction of $E$, both the size and the number  $m_j$ of the generating disks. However, this change is made not only from one step to the next, as in \cite{ACMOU}, but also within the same step of the construction.\\
\\
\noindent
It is instructive to recall the following elementary Lemma in \cite{Ur}, which we prove for the reader's convenience. 

\begin{lemma}\label{FillingAreaOfDiskWithDisks}
Let $\D = \{z \in \C : |z|<1 \}$.
\begin{enumerate}
\item[(a)] There exists an absolute constant $\varepsilon_0 >0$ such that for any $0<R<1$, and any collection of disks $D_j \subset \D$ with disjoint interiors, with radii $r_j=R$, $\mid \cup_j D_j  \mid < (1-\varepsilon_0) \mid \D \mid$.
\item[(b)] For any $\varepsilon >0$, $\delta >0$, there exists a finite collection of disks $D_j \subset \D$ with radii $0< r_j < \delta$ with disjoint interiors (or even disjoint closures), such that $\mid \cup_j D_j  \mid > (1-\varepsilon) \mid \D \mid$.
\end{enumerate}
\end{lemma}

\begin{proof}
Part (a) follows readily from the observation that given any 3 pairwise tangent disks $D_1, D_2, D_3$ with the same radius $R$, in the space they leave between them (i.e. in the bounded component of $\displaystyle \C \setminus \bigcup_{j=1}^{3} D_j$) one can fit another disk $B$, tangent to $D_1, D_2$ and $D_3$, with radius $cR$, where $c$ is an absolute constant independent of $R$.\\
\\
Part (b) follows from Vitali's covering theorem, but we will prove it directly since we will later use some elements from the proof. Given a bounded open set $\Omega$, consider a mesh of squares of side $\delta$. Select those squares entirely contained in the open set, i.e. $\overline{Q_j} \subset \Omega$, say such a collection is $\{ Q_j  \}_{j=1}^{N}$. Then $\displaystyle \mid \Omega \setminus \bigcup_{j=1}^{N} Q_j \mid $ is as small as we wish if $\delta$ is sufficiently small.\\
For each $Q_j$, let $D_j$ be the largest disk inscribed inside it. (Shrink the $D_j$ slightly so that they have disjoint closures.) Then $\mid D_j \mid > \frac{1}{2} \mid Q_j \mid$.\\
Consequently, given $\Omega_0 = \D$, pick a first collection of disks $\ds \{D^1_j \}_{j=1}^{N}$ eating up at least, say, $\frac{1}{10}$ of the area of $\D$. Let $\ds \Omega_1 = \D \setminus \bigcup_{j=1}^{N} D^1_j$, which has area $< \frac{9}{10} \mid \Omega_0 \mid$. Repeat the construction in $\Omega_1$ and so on. The Lemma follows since $\left( \frac{9}{10} \right)^n \longrightarrow 0$ as $n \longrightarrow \infty$.
\end{proof}

Hence, by the above Lemma, in order to fill a very big proportion of the area of the unit disk $\D$ with smaller disks we are forced to consider disks of different radii.
This creates a number of technical complications as we will see later.

{\bf{Step 1}}. Choose first $m_{1,1}$ disjoint disks $D(z_{1,1}^i,R_{1,1}) \subset \D$, $i=1,...,m_{1,1}$, and then $m_{1,2}$ disks $D(z_{1,2}^i,R_{1,2}) \subset \D$, $i=1,...,m_{1,2}$, disjoint among themselves and with the previous ones, and then $m_{1,3}$ disks $D(z_{1,3}^i,R_{1,3}) \subset \D$, $i=1,...,m_{1,3}$, disjoint among themselves and with the previous ones, and so on up to $m_{1,l_1}$ disks $D(z_{1,l_1}^i,R_{1,l_1}) \subset \D$, $i=1,...,m_{1,l_1}$, disjoint among themselves and with the previous ones,
so that they cover a big proportion of the unit disk $\D$ (see Lemma \ref{FillingAreaOfDiskWithDisks}), say $(1-\varepsilon_1)|\D|$. Then, we have that
\begin{equation}\label{AreaCoveredInFirstStep}
c_1:=m_{1,1}\,(R_{1,1})^2 +  m_{1,2}\,(R_{1,2})^2 + ...+ m_{1,l_1}\,(R_{1,l_1})^2 = 1-\varepsilon_1 
\end{equation}
where $0< \varepsilon_1 <1$ is a very small parameter to be chosen later. By the proof of Lemma \ref{FillingAreaOfDiskWithDisks}, we can assume that all radii $R_{1,j} < \delta_1$, for $j=1, ..., l_1$, for a $\delta_1 >0$ as small as we wish.\\
Now to each $j=1,...,l_1$ we will associate a number $0<\sigma_{1,j}<\frac{1}{100}$ to be determined later. \\
%
Let  $r_{1,j}=R_{1,j}$ for $j=1, ..., l_1$. For each $i=1,\dots, m_j$, let $\varphi^i_{1,j}(z)=z^i_{1,j}+(\sigma_{1,j})^K R_{1,j}\,z$ and, using the notation $\alpha D(z,\rho):= D(z,\alpha \rho)$, set
$$\aligned
D^i_j&:=\frac{1}{(\sigma_{1,j})^K}\,\varphi^i_{1,j}(\D)=D(z^i_{1,j}, r_{1,j})\\
(D^i_j)'&:=\varphi^i_{1,j}(\D)=D(z^i_{1,j},(\sigma_{1,j})^K r_{1,j}) \subset D^i_j
\endaligned$$ 
As the first approximation of the mapping we define
$$
 g_1(z)=
\begin{cases}
(\sigma_{1,j})^{1-K}(z-z^i_{1,j})+z^i_{1,j}, &z\in (D^i_j)'\\
\left|\frac{z-z^i_{1,j}}{r_{1,j}}\right|^{\frac{1}{K}-1}(z-z^i_{1,j})+z^i_{1,j}, \; &z\in D^i_j\setminus (D^i_j)'\\
z, & z \notin \cup D^i_j
\end{cases}
$$
This is a $K$-quasiconformal mapping, conformal outside of $\ds \bigcup_{j=1}^{l_1} \bigcup_{i=1}^{m_{1,j}}(D^i_j \setminus (D^i_j)' )$. It maps each $D^i_j$ onto itself and $(D^i_j)'$ onto $(D^i_j)''=D(z^i_{1,j},\sigma_{1,j} \: r_{1,j})$, while the rest of the plane remains fixed. Write $\phi_1=g_1$. 
\\
{\bf{Step 2}}. We have already fixed $l_1, m_{1,j}, R_{1,j}, \sigma_{1,j}$ and $c_1$. Choose now $m_{2,1}$ disjoint disks $D(z_{2,1}^n,R_{2,1}) \subset \D$, $n=1,...,m_{2,1}$, and then $m_{2,2}$ disks $D(z_{2,2}^n,R_{2,2}) \subset \D$, $n=1,...,m_{2,2}$, disjoint among themselves and with the previous ones (within this second step), and then $m_{2,3}$ disks $D(z_{2,3}^n,R_{2,3}) \subset \D$, $n=1,...,m_{2,3}$, disjoint among themselves and with the previous ones (within this second step), and so on up to $m_{2,l_2}$ disks $D(z_{2,l_2}^n,R_{2,l_2}) \subset \D$, $n=1,...,m_{2,l_2}$, disjoint among themselves and with the previous ones (within this second step), so that they cover a big proportion of the unit disk $\D$, for instance $(1-\varepsilon_2)|\D|$ (again by Lemma \ref{FillingAreaOfDiskWithDisks}.) Then, we have that
\begin{equation}\label{AreaCoveredInSecondStep}
c_2:=m_{2,1}\,(R_{2,1})^2 +  m_{2,2}\,(R_{2,2})^2 + ...+ m_{2,l_2}\,(R_{2,l_2})^2 = 1-\varepsilon_2 
\end{equation}
and $0< \varepsilon_2 <1$ will be chosen later. As in the previous step, we can assume that all radii $R_{2,k} < \delta_2$, for $k=1, ..., l_2$, for a $\delta_2 >0$ as small as we wish.\\
\\
Repeating the above procedure, consider now the parameters $\sigma_{2,k} >0$, which we will associate to each one of the disks $D(z_{2,k}^n,R_{2,k})$, with $k=1, ..., l_2$, and all possible values of $n$. We associate the same parameter $\sigma_{2,k}$ to all the disks of the form $D(z_{2,k}^n,R_{2,k})$ (so $\sigma_{2,k}$ does not depend on $n$.) The parameters $\sigma_{2,k}$ will be chosen later, and they will all be small, say $\sigma_{2,k} < \frac{1}{100}$ for $k=1, ..., l_2$.\\
Denote $r_{\{2,k\},\{1,j\}}=R_{2,k}\,\sigma_{1,j} \: r_{1,j}$ and $\varphi^n_{2,k}(z)=z^n_{2,k}+(\sigma_{2,k})^K R_{2,k}\,\,z$, \, and define the auxiliary disks
$$\aligned
D_{j,k}^{i,n}=\phi_1\left(\frac{1}{(\sigma_{2,k})^K}\, \varphi^{i}_{1,j} \circ \varphi^{n}_{2,k}(\D)\right)=D(z^{i,n}_{j,k}, r_{\{2,k\},\{1,j\}})\\
(D_{j,k}^{i,n})'=\phi_1\left(\, \varphi^{i}_{1,j} \circ \varphi^{n}_{2,k}(\D)\right)=D(z^{i,n}_{j,k}\, , (\sigma_{2,k})^K r_{\{2,k\},\{1,j\}})
\endaligned$$
for certain $z^{i,n}_{j,k} \in \D$, where $i=1,\dots,m_{1,j}$, $n=1,\dots,m_{2,k}$, $j=1,\dots,l_1$ and $k=1,\dots,l_2$. Now let
$$
g_2(z)=
\begin{cases}
(\sigma_{2,k})^{1-K}(z-z^{i,n}_{j,k})+z^{i,n}_{j,k}&z\in (D_{j,k}^{i,n})'\\
\left|\frac{z-z^{i,n}_{j,k}}{r_{\{2,k\},\{1,j\}}}\right|^{\frac{1}{K}-1}(z-z^{i,n}_{j,k})+z^{i,n}_{j,k}&z\in D_{j,k}^{i,n}\setminus (D_{j,k}^{i,n})'\\
z&\text{otherwise}
\end{cases}
$$
Clearly, $g_2$ is $K$-quasiconformal, conformal outside of $\ds \bigcup_{i,j,k,n} \left( D_{j,k}^{i,n} \setminus (D_{j,k}^{i,n})' \right)$, maps each $D_{j,k}^{i,n}$ onto itself and $(D_{j,k}^{i,n})'$ onto $(D_{j,k}^{i,n})''=D(z^{i,n}_{j,k},\: \sigma_{2,k}\: r_{\{2,k\},\{1,j\}})$, while the rest of the plane remains fixed. 
Define $\phi_2=g_2\circ\phi_1$.\\


In the picture below the size of the parameters $\sigma$ has been greatly magnified for the convenience of the reader (so that e.g. the annuli $D_{j}^{i}\setminus (D_{j}^{i})'$ and their images under $\phi$ are much thinner in the picture than in the proof.)



\begin{figure}[ht]
\begin{center}
\includegraphics{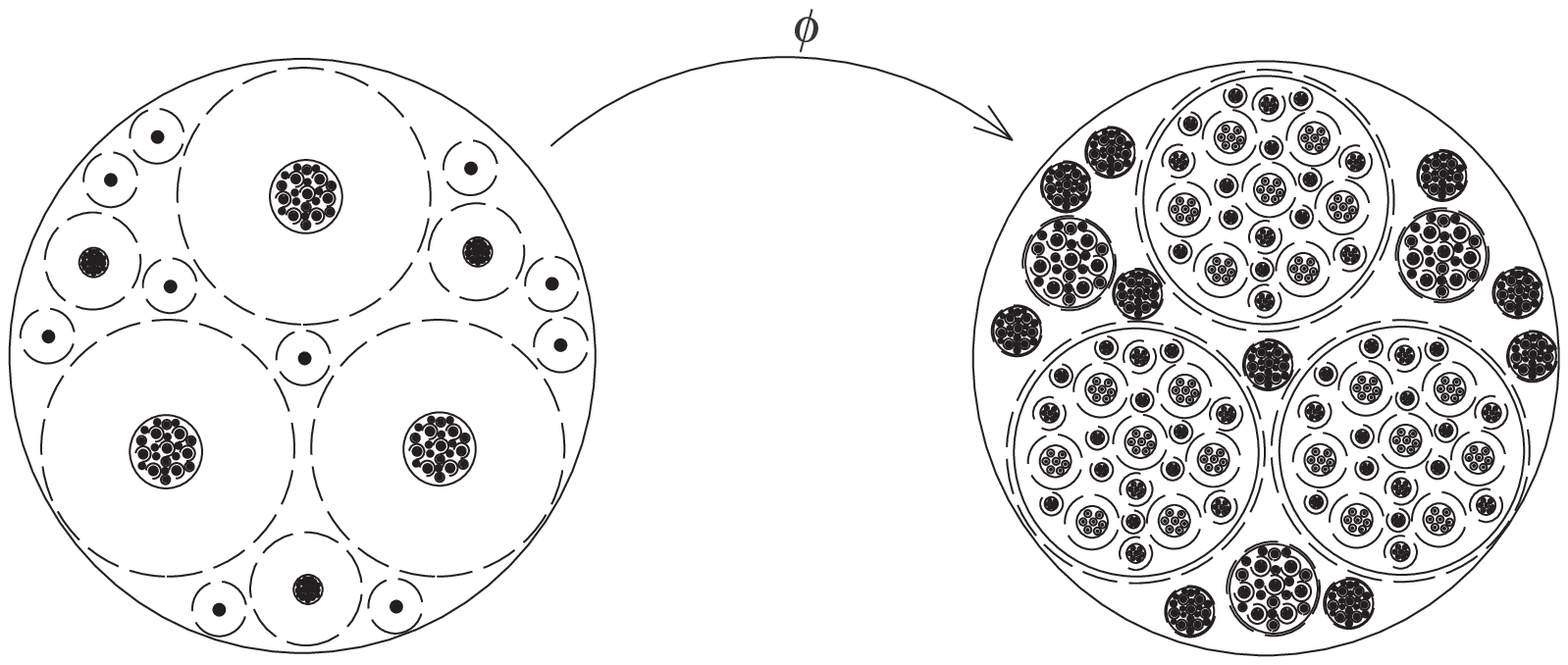}
\end{center}
\end{figure}

\noindent  
{\bf{The induction step}}. After step $N-1$ we take $m_{N,1}$ disjoint disks $D(z_{N,1}^q,R_{N,1}) \subset \D$, $q=1,...,m_{N,1}$, and then $m_{N,2}$ disks $D(z_{N,2}^q,R_{N,2}) \subset \D$, $q=1,...,m_{N,2}$, disjoint among themselves and with the previous ones (within this $N^{th}$ step), and then $m_{N,3}$ disks $D(z_{N,3}^q,R_{N,3}) \subset \D$, $q=1,...,m_{N,3}$, disjoint among themselves and with the previous ones (within this $N^{th}$ step), and so on up to $m_{N,l_N}$ disks $D(z_{N,l_N}^q,R_{N,l_N}) \subset \D$, $q=1,...,m_{N,l_N}$, disjoint among themselves and with the previous ones (within this $N^{th}$ step), so that they cover a big proportion of the unit disk $\D$. Then, we have that
\begin{equation}\label{AreaCoveredInNthStep}
c_N:=m_{N,1}\,(R_{N,1})^2 +  m_{N,2}\,(R_{N,2})^2 + ...+ m_{N,l_N}\,(R_{N,l_N})^2 = 1-\varepsilon_N 
\end{equation}
where $0< \varepsilon_N <1$ is a very small parameter to be chosen later. Again, we can assume that all the radii $R_{N,p} < \delta_N$, for $p=1, ..., l_N$, and for a $\delta_N >0$ as small as we wish.\\
Repeating the above procedure, consider now the parameters $\sigma_{N,p} >0$, which we will associate to each one of the disks $D(z_{N,p}^q,R_{N,p})$, with $p=1, ..., l_N$, and all possible values of $q$. We associate the same parameter $\sigma_{N,p}$ to all the disks of the form $D(z_{N,p}^q,R_{N,p})$ (so the parameter $\sigma_{N,p}$ does not depend on $q$.) The parameters $\sigma_{N,p}$ will be chosen later, and they will all be quite small, say $\sigma_{N,p} < \frac{1}{100}$ for $p=1, ..., l_N$.\\
Denote then 
$r_{ \{N,p\}, \{N-1,h\}, \ldots , \{2,k\}, \{1,j\} }=R_{N,p}\,\,\sigma_{N-1,h}\,\, 
r_{ \{N-1,h\}, \ldots , \{2,k\}, \{1,j\} }$, and $\varphi^q_{N,p}(z)=z^q_{N,p}+ (\sigma_{N,p})^K \, R_{N,p}\,z$. For any multiindexes $I=(i_1,...,i_N)$ and  $J=(j_1,...,j_N)$, where $1\leq i_k\leq m_{k,j_k}$, $1\leq j_k\leq l_k$, and $k=1,...,N$, let 
\begin{equation}\label{FormulaDIJAndDIJPrime}
\aligned
D^{I}_{J}=\phi_{N-1}\left(\frac{1}{(\sigma_{N,p})^K}\, \varphi^{i_1}_{1,j_1} \circ \dots \circ \varphi^{i_N}_{N,j_N}(\D) \right) = D\left(z^{I}_{J}, r_{ \{N,p\}, \{N-1,h\}, \ldots , \{2,k\}, \{1,j\} } \right)\\
(D^{I}_{J})'=\phi_{N-1}\left( \varphi^{i_1}_{1,j_1}  \circ \dots \circ \varphi^{i_N}_{N,j_N}(\D) \right) = D\left(z^{I}_{J}, (\sigma_{N,p})^K \, r_{ \{N,p\}, \{N-1,h\}, \ldots , \{2,k\}, \{1,j\} }\right)
\endaligned
\end{equation}
and let
\begin{equation}\label{FormulaGN}
g_N(z)=
\begin{cases}
(\sigma_{N,p})^{1-K}(z-z^{I}_{J})+z^{I}_{J}&z\in (D^{I}_{J})'\\
\left|\frac{z-z^{I}_{J}}{r_{ \{N,p\}, \{N-1,h\}, \ldots , \{2,k\}, \{1,j\} }}\right|^{\frac{1}{K}-1}(z-z^{I}_{J})+z^{I}_{J}&z\in D^{I}_{J}\setminus (D^{I}_{J})'\\
z&\text{otherwise}
\end{cases}
\end{equation}
Clearly, $g_N$ is $K$-quasiconformal, conformal outside of $\ds \bigcup_{\substack{I=(i_1,...,i_N)\\J=(j_1,...,j_N)}} \left( \, D^{I}_{J}\setminus (D^{I}_{J})' \, \right)$, maps $D^{I}_{J}$ onto itself and $(D^{I}_{J})'$ onto $(D^{I}_{J})''=D \left( z^{I}_{J},\, \sigma_{N,p} \, \, r_{ \{N,p\}, \{N-1,h\}, \ldots , \{2,k\}, \{1,j\} } \right)$, while the rest of the plane remains fixed. 
Now define $\phi_N=g_N\circ\phi_{N-1}$.\\  
\\
Since each $\phi_N$ is $K$-quasiconformal and  equals the identity  outside the unit disk $\D$, there exists a limit $K$-quasiconformal mapping  
$$\phi=\lim_{N\to\infty}\phi_N$$
with convergence in $W^{1,p}_{loc}(\C)$ for any $p<\frac{2K}{K-1}$. On the other hand, $\phi$ maps the compact set 
\begin{equation}\label{EquationForCompactSetEInSource}
E=\displaystyle\bigcap_{N=1}^\infty\left(
\ds \bigcup_{\substack{i_1,...,i_N \\ j_1,...,j_N}}
\varphi^{i_1}_{1,j_1}  \circ \dots \circ \varphi^{i_N}_{N,j_N} \left( \, \overline{\D} \, \right)
\right)
\end{equation}
to the compact set
\begin{equation}\label{EquationForCompactSetPhiEInTarget}
\phi(E)=\bigcap_{N=1}^\infty\left(
\ds \bigcup_{\substack{i_1,...,i_N \\ j_1,...,j_N}}
\psi^{i_1}_{1,j_1}  \circ \dots \circ \psi^{i_N}_{N,j_N} \left( \, \overline{\D} \, \right)
\right)
\end{equation}
where we have written $\psi^{i_k}_{k,j_k}(z)=z^{i_k}_{k,j_k} + \sigma_{k,j_k} \, R_{k,j_k} \, z$, and where $1\leq i_k\leq m_{k,j_k}$, $1\leq j_k\leq l_k$, and $k \in \N$. \\
\\
Notice that with this notation, a building block in the $N^{th}$ step of the construction of $E$ (i.e. a set of the type $\varphi^{i_1}_{1,j_1}  \circ \dots \circ \varphi^{i_N}_{N,j_N} \left( \, \overline{\D} \, \right)$) is a disk with radius given by 
\begin{equation}\label{RadiusSourceNthStep}
s_{j_1,...,j_N}=\left( (\sigma_{1,j_1})^K \, R_{1,j_1} \right) \dots \left( (\sigma_{N,j_N})^K R_{N,j_N} \right) 
\end{equation}
and a building block in the $N^{th}$ step of the construction of $\phi(E)$ (i.e. a set of the type $\psi^{i_1}_{1,j_1}  \circ \dots \circ \psi^{i_N}_{N,j_N} \left( \, \overline{\D} \, \right)$) is a disk with radius given by 
\begin{equation}\label{RadiusTargetNthStep}
t_{j_1,...,j_N}=\left( \sigma_{1,j_1} \, R_{1,j_1} \right) \dots \left( \sigma_{N,j_N} \, R_{N,j_N} \right) .
\end{equation}
As is explained in \cite{Ur}, the key now is the right choice of parameters. So we choose $\sigma_{k,j_k}$ satisfying
\begin{equation}\label{ChoiceOfSigmaAsFunctionOfR}
(\sigma_{k,j_k})^{tK} = (R_{k,j_k})^{2-t}
\end{equation}
for all possible values of $k$ and $j_k$. 
The choice \eqref{ChoiceOfSigmaAsFunctionOfR} actually has some geometric meaning related to area. Namely, forgetting about subindexes, 
\begin{equation}\label{SimplificationOfTargetAndSourceEquationsToArea}
\left( \sigma^K \, R \right)^t = \left( \sigma \, R \right)^{ \frac{2Kt}{2+(K-1)t} } = \left( \sigma \, R \right)^{t'} = R^2.
\end{equation}
which is helpful when dealing with the sums involved in the calculations of $\H^t(E)$ and of $\H^{t'}(\phi(E))$ (i.e. sums of the type $\sum \left( s_{j_1,...,j_N}  \right)^t$ and $\sum \left( t_{j_1,...,j_N}   \right)^{t'}$, respectively.) \\
As in \cite{Ur}, we choose $\varepsilon_n \rightarrow 0$ so fast that 
\begin{equation}\label{ChoiceOfVarepsilonSoThatProductOfAreasConverges}
\prod_{n=1}^{\infty} \left( 1- \varepsilon_n \right) \approx 1 .
\end{equation}
With such a choice of parameters, it is proved in \cite{Ur} that $\phi$ is $K$-quasiconformal and that
\begin{equation}\label{UpperAndLowerBoundsForHausdorffMeasureInSourceAndTarget}
0 < \H^{t}(E) < \infty \, \, \text{ and that } \, \, 
0 < \H^{t'}(\phi(E))  < \infty . 
\end{equation}
This finishes the sketch of proof of Theorem \ref{TheoremConstructionForUsualHausdorffMeasures}.
\end{proof}

Let us make some remarks which will be useful later.\\
\\
Fix a building block $D$ at scale $N-1$ for $E$, i.e. let $D = \varphi^{i_1}_{1,j_1}  \circ \dots \circ \varphi^{i_{N-1}}_{N-1,j_{N-1}} \left( \, \overline{\D} \, \right)$ for some choice of $i_k$ and $j_k$, $1 \leq k \leq N-1$. As usual, the {\it{children}} of $D$ are the building blocks at scale $N$ contained in $D$, that is, the disks of the form 
$$D' = \varphi^{i_1}_{1,j_1}  \circ \dots \circ \varphi^{i_{N-1}}_{N-1,j_{N-1}} \circ \varphi^{i_{N}}_{N,j_{N}}  \left( \, \overline{\D} \, \right),$$ 
for any choice of $i_{N}$ and $j_{N}$, but with the same choices of $i_k$ and $j_k$ for $1 \leq k \leq N-1$ as for $D$. The genealogical terminology (parents, cousins, descendants, generation, etc.) has the obvious meaning in this context.\\
\\
For any multiindexes $I=(i_1,...,i_N)$ and  $J=(j_1,...,j_N)$, where $1\leq i_k\leq m_{k,j_k}$, $1\leq j_k\leq l_k$, and $k=1,...,N$, we will denote by
\begin{equation}\label{DefinitionProtectingDisk}
P^{N}_{I;J} = \frac{1}{(\sigma_{N,j_N})^K}\, \varphi^{i_1}_{1,j_1} \circ \dots \circ \varphi^{i_N}_{N,j_N}(\D)
\end{equation}
a {\it{protecting}} disk of generation $N$. Then, $P^N_{I;J}$ has radius 
$$r(P^{N}_{I;J}) = \frac{1}{\left( \sigma_{N,j_N} \right)^K} s_{j_1,...,j_N}=\left( \sigma_{1,j_1} \, \dots \sigma_{N-1,j_{N-1}} \right)^K  \left(   R_{1,j_1} \dots   R_{N,j_N} \right).$$ 
Analogously, we will write
\begin{equation}\label{DefinitionGeneratingDisk}
G^{N}_{I;J} = \varphi^{i_1}_{1,j_1}  \circ \dots \circ \varphi^{i_N}_{N,j_N}(\D)
\end{equation}
in order to denote a {\it{generating}} disk of generation $N$, which has radius $$r(G^{N}_{I;J}) = s_{j_1,...,j_N}=\left( \sigma_{1,j_1} \, \dots \sigma_{N,j_N} \right)^K  \left(   R_{1,j_1} \dots   R_{N,j_N} \right).$$
With this notation, (see \eqref{FormulaDIJAndDIJPrime}), we have $ D^{I}_{J}=\phi_{N-1}\left( P^{N}_{I;J}  \right) $, $ (D^{I}_{J})'=\phi_{N-1}\left( G^{N}_{I;J}  \right) $, and $ (D^{I}_{J})''= \phi_{N} \left( G^{N}_{I;J}  \right) $. Notice that, except for the closure, the disks $G^{N}_{I;J}$ are what we called the building blocks above. 
We will also refer to the unit disk $\D$ as $G^{0}$ and $\phi_0$ will be the identity map. We will mostly refer to $G^{N}_{I;J}$ and $P^{N}_{I;J}$ as open disks (as opposed to their closure), unless the context suggests differently.

\section{The calculation of the H\"{o}lder exponent of $\phi$}\label{CalculationHolderExponentPhi}

The main purpose of this section is to prove the following result.
\begin{thm}\label{TheoremYieldingHolderExponentPhi}
The $K$-quasiconformal mapping $\phi$ from Theorem \ref{TheoremConstructionForUsualHausdorffMeasures} is locally H\"older continuous with exponent $t/t'$.
\end{thm}
\noindent
By the Poincar\'e inequality together with the quasiconformality of $\phi$, \cite[p.64]{G} it is enough 
to show that for any disk $D$ with, say, $\diam(D) \lesssim 1$, 
\begin{equation}\label{IntegralHolderCondition}
\int_DJ(z,\phi)\,dA(z)\leq C\,\diam(D)^{2t/t'}.
\end{equation}
In order to prove \eqref{IntegralHolderCondition}, we will need several lemmas. \\
\\
An easy consequence of quasisymmetry is that the Jacobian of a $K$-quasiconformal mapping is a doubling measure, with doubling constant only depending on $K$, i.e. $\int_DJ(z,\phi)\,dA(z) \approx \int_{2D}J(z,\phi)\,dA(z)$. A further easy consequence of this fact is the following

\begin{lemma}\label{ConsequenceJacobianDoublingMeasure}
Let $C>0$ be given. Assume that $\frac{1}{C} \leq \alpha \leq C$ and $\beta \in \C$ be such that $|\beta| \leq C$. Then, for any $K$-quasiconformal mapping $\phi$, and any disk $D$ of radius $r(D)$,
\begin{equation}\label{FormulaConsequenceJacobianDoublingMeasure}
\int_{D(a,r)}J(z,\phi)\,dA(z) \approx \int_{D(a+\beta r, \alpha r)}J(z,\phi)\,dA(z),
\end{equation}
with constants that depend only on $K$ and $C$. 
\end{lemma}
\noindent
As a consequence, it will be sufficient to prove \eqref{IntegralHolderCondition} only for disks $D$ strictly included in $\D$, since $\phi$ restricted to $\C \setminus \D$ is the identity map. 
\begin{proof}
Apply the doubling condition to $D' = D(z',R') = D(a+\beta r, \alpha r) \subset D(a, 2Cr)$, and to $D(a,r) = D(z' - \frac{\beta}{\alpha}R', \frac{1}{\alpha}R') \subset D(z',\left( C^2 + C \right) R')$. 
\end{proof}

\begin{lemma}\label{JacobianOfGN}
The Jacobian of $g_N$ 
is given by
\begin{equation}\label{FormulaJacobianOfGN}
J(z,g_N)= \begin{cases}
\left( (\sigma_{N,p})^{1-K}  \right)^2  &z\in (D^{I}_{J})'\\
\frac{1}{K}  \left|\frac{z-z^{I}_{J}}{ \left( \sigma_{1,j} \ldots  \sigma_{N-1,h} \right) \left( R_{1,j} \ldots  R_{N,p} \right) }\right|^{ 2\left( \frac{1}{K}-1 \right) }  
&z\in D^{I}_{J}\setminus (D^{I}_{J})'\\
1&\text{otherwise}
\end{cases}
\end{equation}
\end{lemma}
\begin{proof}
This comes from direct calculations and equations \eqref{FormulaGN} and \eqref{FormulaDIJAndDIJPrime}.
\end{proof}


\begin{remark}\label{ConsequencesFormulaJacobianGN} As a consequence of  Lemma \ref{JacobianOfGN}, we note that 
\begin{itemize}
\item[(a)] $J(z,g_N)$ is radial in $D^{I}_{J}$ with respect to the center $z^{I}_{J}$.
\item[(b)] $J(z,g_N)$ is radially decreasing in $D^{I}_{J}\setminus (D^{I}_{J})'$. 
\item[(c)] $J(z,g_N)$ is radially nonincreasing in $D^{I}_{J}$.
\end{itemize}
\end{remark}
\noindent
We will reduce some of the cases appearing in the proof of \eqref{IntegralHolderCondition} to the following

\begin{lemma}\label{CheckingIntegralHolderConditionForPhiNForDisksContainedAndConcentricToPN}
Let $D$ be a disk contained in $P^{N}_{I;J}$, for some $P^{N}_{I;J}$.
\begin{enumerate}
\item[(a)] If $D \subseteq G^{N}_{I;J}$, there exists a constant $C>0$, independent of $D$ and $N$, such that (see \eqref{IntegralHolderCondition}) $$\int_DJ(z,\phi_N)\,dA(z)\leq\,C\,\diam(D)^\frac{2t}{t'}.$$
\item[(b)] If $D$ is concentric to $P^{N}_{I;J}$, the conclusion in $(a)$ also holds.
\end{enumerate}
\end{lemma}
\begin{proof}
We first prove (a). Let us assume that $D \subseteq G^{N}_{I;J}$. Then,
$$r(D) \leq r(G^{N}_{I;J})=s_{j_1,...,j_N}=\sigma^K\,R$$
where we have written $\sigma=\sigma_{1,j_1} \, \dots \sigma_{N,j_N}$ and $R=R_{1,j_1} \dots   R_{N,j_N}$. Notice that $J(\cdot,\phi_N)$ is constant on $G^{N}_{I;J}$, so that we do not actually need $D$ to be concentric to $P^{N}_{I;J}$. An iteration of Lemma \ref{JacobianOfGN} and \eqref{FormulaDOverDPrime} give that
\begin{equation}\label{DContainedInGeneratorForPhiNcalculation1}
\int_D J(z,\phi_N)\,dA(z) 
= \sigma^{2\left( 1-K \right)} |D| 
= |D|^{\frac{t}{t'}} \sigma^{2\left( 1-K \right)} |D|^{ \frac{K-1}{K} \left( 1-\frac{t}{2} \right) }
= |D|^{\frac{t}{t'}} \left\{\frac{ |D|^{1-\frac{t}{2}} }{ \sigma^{2K}}  \right\}^{\frac{K-1}{K}}.
\end{equation}
Using \eqref{ChoiceOfSigmaAsFunctionOfR}, we can see that the term in braces in \eqref{DContainedInGeneratorForPhiNcalculation1} satisfies
\begin{equation}\label{DContainedInGeneratorForPhiNcalculation2}
\frac{ |D|^{1-\frac{t}{2}} }{ \sigma^{2K} } \approx 
\frac{ \left( \diam D  \right)^{2-t} }{\sigma^{2K} } 
\lesssim \frac{  \left(\sigma^K\,R \right)^{2-t} }{ \sigma^{2K} } = 
\frac{R^{2-t} }{ \sigma^{tK} } = 1.
\end{equation}
For the proof of (b), let us assume now that $ G^{N}_{I;J} \subseteq D \subseteq P^{N}_{I;J}$
. Then, 
\begin{equation}\label{BoundsForRDIfDInsideProtector}
\sigma^K\,R\leq r(D) \leq \frac{\sigma^K\,R}{\sigma_{N,j_{N}}^K}.
\end{equation}
Since $\phi_N$ is radial inside $P^{N}_{I;J}$
, and $\phi_N=g_N\circ\phi_{N-1}$, then it may be easily checked that $\phi_N (D)$ is a disk of radius
\begin{equation}\label{CalculationRadiusPhiNDIfDInsideProtector}
r(\phi_N (D)) 
=  r(D)^{ \frac{1}{K} } R^{ 1-\frac{1}{K} }.
\end{equation}
Hence, by \eqref{BoundsForRDIfDInsideProtector} and \eqref{ChoiceOfSigmaAsFunctionOfR},
\begin{equation}\label{EstimateOfQuotientOfAreaOfImageByAreaOfSourceRaisedToTOverTPrimeIfDInsideProtector}
\aligned
\left\{ \frac{ \int_D J(z,\phi_N)\,dA }{ |D|^{\frac{t}{t'}} } \right\}^{\frac{1}{2}} 
& \approx  \frac{ r(\phi_N (D)) }{ r(D)^{\frac{t}{t'}} }
=\frac{ r(D)^{\frac{1}{K}}\, R^{ 1-\frac{1}{K} } } { r(D)^{\frac{1}{K} + \frac{K-1}{2K}t  } } = \left(  \frac{R}{r(D)^{ \frac{t}{2} } }  \right)^{ \frac{K-1}{K} } \leq \\
& \leq \left( \frac{R}{ \left( \sigma^K\,R \right)^{\frac{t}{2}}} \right)^{ \frac{K-1}{K} } 
= \left( R^{1-\frac{t}{2}}\,\sigma^{-\frac{Kt}{2}}  \right)^{ \frac{K-1}{K} } = 1.
\endaligned
\end{equation}
\end{proof}

We will also make use of the following elementary geometric fact.

\begin{lemma}\label{IfDiskBDoesNotMeetMuchOfOtherDisksThenHalfBDoesNotMeetOtherDisks}
Let $B$ and $\{D_{i} \}_{i=1}^{n}$ be disks. Assume that $|B \cap D_i| \leq \frac{1}{100} |B|$ for all $i = 1, \dots ,n$, and that $D_i \nsubseteq B $ for all $i = 1, \dots ,n$. Then $\frac{1}{2}B \cap D_i = \emptyset$ for all $i = 1, \dots ,n$. 
\end{lemma}
\begin{proof}
If that were not the case, then $\frac{1}{2}B \cap D_{i_0} \neq \emptyset$, for some $i_0$. Consider a disk $D' \subseteq B \cap D_{i_0}$ with $r(D') = \frac{1}{4} r(B)$ (e.g. if $D'$ is inner tangent to $B$.) Then $|B \cap D_{i_0}| \geq | B \cap D' | = |D'| = \frac{1}{16} |B| > \frac{1}{100} |B|$, a contradiction.
\end{proof}

\noindent
For the proof of \eqref{IntegralHolderCondition}, we first notice that there are disks $D$ that intersect infinitely many protecting and generating disks (for this, simply take $D$ such that its boundary has points of the set $E$). 
Because of this, the proof of \eqref{IntegralHolderCondition} will be divided into several cases, in all of which we will assume that $D$ satisfies
\begin{equation}\label{MaximalityHypothesisInN}
D\subseteq G^{N-1}_{I',J'}
\end{equation}
where $N$ is maximum possible. By Lemma \ref{ConsequenceJacobianDoublingMeasure}, $N$ always satisfies $N\geq 1$.

\begin{enumerate}
\item[(1)] {\bf{Case 1:  
$D \cap P^{N}_{I;J} = \emptyset$ for all $I,J$.
}} The case $N=1$ is trivial, since then $\phi (D) = D$. 
If $N>1$, then $\phi (D) = \phi_{N-1} (D)$, and Lemma \ref{CheckingIntegralHolderConditionForPhiNForDisksContainedAndConcentricToPN} (a) applies.

\item[(2)] {\bf{Case 2: 
$D \cap P^{N}_{I;J} \neq \emptyset$ for some $I,J$, but $D \cap G^{N}_{I;J} = \emptyset$ for any $G^{N}_{I;J}$.
}} Let $P^{N}_{{I_k};{J_k}},\,k=1,...,M$ denote the protecting disks of generation $N$ which satisfy that $D \cap P^{N}_{I_k;J_k} \neq \emptyset$. Notice that if $P^{N}_{I;J}$ is a protecting disk such that $D \cap P^{N}_{I;J} \neq \emptyset$ then $P^{N}_{I;J} \nsubseteq D$, because $G^{N}_{I;J} \subset P^{N}_{I;J}$ and $D \cap G^{N}_{I;J} = \emptyset$. We distinguish now two subcases, according to the size of the intersections $D\cap P^{N}_{I_k;J_k}$.

\begin{enumerate}
\item[(2a)] {\bf{Case 2a: $|D \cap P^{N}_{{I_k};{J_k}}| < \frac{1}{100} |D|$ for all $k$.
}} In this case, by Lemma \ref{ConsequenceJacobianDoublingMeasure}, Lemma \ref{IfDiskBDoesNotMeetMuchOfOtherDisksThenHalfBDoesNotMeetOtherDisks}, and Case 1 we have that
\begin{equation}\label{DiskDDoesNotMeetSubstantiallyAnyProtectingDisk}
\int_DJ(z,\phi)\,dA(z) \approx \int_{\frac{1}{2}D}J(z,\phi)\,dA(z) \lesssim |D|^{\frac{t}{t'}}.
\end{equation}

\item[(2b)] {\bf{Case 2b: There exists $k_0$ such that $|D \cap P^{N}_{{I_{k_0}};{J_{k_0}}}| \geq \frac{1}{100} |D|$.
}} In this case we necessarily have that $r( P^{N}_{{I_{k_0}};{J_{k_0}}} ) \geq \frac{1}{10} r(D)$. Otherwise, we would have a contradiction since $|D \cap P^{N}_{{I_{k_0}};{J_{k_0}}}| \leq |P^{N}_{{I_{k_0}};{J_{k_0}}} | < \frac{1}{100} |D|$. Thus, consider a disk $D' \subset P^{N}_{{I_{k_0}};{J_{k_0}}}$, inner tangent to $P^{N}_{{I_{k_0}};{J_{k_0}}}$, with radius $r(D') = \frac{1}{100} r(D)$, such that $D'\cap D \neq \emptyset$.
%
Let $D''$ be the disk of radius $r(D'') = r(D')$, concentric to $P^{N}_{{I_{k_0}};{J_{k_0}}}$. 
By Lemma \ref{ConsequenceJacobianDoublingMeasure}, $|\phi(D)|\simeq|\phi(D')|$. Since $D'\cap G^{N}_{{I_{k_0}};{J_{k_0}}}=\emptyset$, we get that $\phi(D')=\phi_N(D')$ as sets. But $D' \cap D'' = \emptyset$, so by Remark \ref{ConsequencesFormulaJacobianGN} and Lemma \ref{CheckingIntegralHolderConditionForPhiNForDisksContainedAndConcentricToPN} we get
$$|\phi(D)|\simeq\int_{D'}J(z,\phi_N)\leq\int_{D''}J(z,\phi_N)\leq C\,|D''|^{t/t'}\simeq C\,|D|^{t/t'}.$$

\end{enumerate}

\item[(3)] {\bf{Case 3: $D \cap G^{N}_{I;J} \neq \emptyset $ for exactly one disk $G^{N}_{I;J}$ (and not more.)
}} First of all, notice that $D$ will not be included in $G^{N}_{I;J}$  (although they have nonempty intersection) because from \eqref{MaximalityHypothesisInN} we know that $N$ is maximal. Let $P^N_{I;J}$ be the protecting disk corresponding to $G^{N}_{I;J}$. We distinguish three cases:
\begin{itemize}
\item[(a)] If $r(D)<r(G^N_{I;J})$, then we use Lemma \ref{ConsequenceJacobianDoublingMeasure} to replace $D$ by $D'$. Here $D'$ is obtained by translating $D$ not more than a distance $2r(D)$, so that $D'\cap G^N_{I;J}=\emptyset$. Now we are led to Case 2 above with the same $N$ for $D'$ and $D$ since $D' \subset P^N_{I;J} \subset G^{N-1}_{I',J'}$. 
\item[(b)] If $r(G^N_{I;J})\leq r(D)<r(P^N_{I;J})$, then we can translate $D$ to get a new disk $D'$ concentric with $P^N_{I;J}$, such that $r(D)=r(D')$. Then use Lemma \ref{ConsequenceJacobianDoublingMeasure} and Lemma \ref{CheckingIntegralHolderConditionForPhiNForDisksContainedAndConcentricToPN}.
\item[(c)] If $r(D)\geq r(P^N_{I,J})$, then we replace $D$ by $D' \subset D \subset G^{N-1}_{I',J'}$, where $r(D')=\frac{1}{10}\,r(D)$ and $D'$ does not meet any generating disk of $N$-th generation. Now we are led again to Case 2 or Case 1.
\end{itemize}

\item[(4)] {\bf{Case 4: $D$ meets at least $2$ different $G^N_{I,J}$.}} This is the most complicated situation, and its proof is given in Lemma \ref{IfDiskDmeetsAtLeast2GeneratorsGNAndDIsContainedInGNMinus1} below.
\end{enumerate}
\noindent
For the proof of Lemma \ref{IfDiskDmeetsAtLeast2GeneratorsGNAndDIsContainedInGNMinus1}, we will make use of the following interesting fact.

\begin{lemma}\label{IfDMeets2OrMoreGeneratingDisksThen4DContainsTheCorrespondingProtectingDisks}
Let $D$ be a disk. Let $\left\{G_i\right\}_{i=1}^{m}$ denote the collection of generating disks $G_i=G^{N}_{I_i;J_i}$, of generation $N$, such that $D \cap G_i \neq \emptyset$. Assume that $m \geq 2$. 
Then 
\begin{equation}\label{FormulaForIfDMeets2OrMoreGeneratingDisksThen4DContainsTheCorrespondingProtectingDisks}
\bigcup_{i=1}^{m} P_i\subseteq 4D.
\end{equation}
where $P_i=P^N_{I_i;J_i}$ is the protecting disk corresponding to $G^N_{I_i;J_i}$.
\end{lemma}
\begin{proof}
This Lemma is proved in \cite{Ur}, but we repeat the
proof for the convenience of the reader. Recall that the parameters $R_{k,j_k}$ are chosen so small that the parameters $\sigma_{k,j_k}$ are also quite small, say $< \, \frac{1}{100}$. 
By hypothesis, $D \cap G_i\neq \emptyset$ for all $i=1, \dots , m $ and $m \geq 2$, 
therefore
\begin{equation}\label{CompareRadiusDWithRadiusPN}
2 \, r(D) \geq \frac{99}{100} \,\, r(P_i)
\end{equation}
for $i=1, \dots , m $, since $G_i$ is a disk concentric to $P_i$, tiny in comparison with $P_i$, and the disks $P_i$ are pairwise disjoint. Consequently, for $i=1, \dots , m $,
\begin{equation}\label{DilatedDCoversPN}
G_i \subset 2D \hspace{1cm}\text{ and }\hspace{1cm}P_i \subset 4D.
\end{equation}
\end{proof}

We finally get to Lemma \ref{IfDiskDmeetsAtLeast2GeneratorsGNAndDIsContainedInGNMinus1} in order to conclude the proof of \eqref{IntegralHolderCondition}.

\begin{lemma}\label{IfDiskDmeetsAtLeast2GeneratorsGNAndDIsContainedInGNMinus1}
Let $B$ be a disk, and let  $G^{N-1}_{I';J'}$ be the smallest generating disk such that $B\subseteq G^{N-1}_{I',J'}$. Assume that $B$ intersects at least two generating disks $G^{N}_{I_i;J_i}$ of $N$-th generation, i.e. $D\cap G^{N}_{I_i;J_i}\neq \emptyset$ for $i=1,2$.
Then \eqref{IntegralHolderCondition} holds for $B$ and the $K$-quasiconformal mapping $\phi$ from Theorem \ref{TheoremConstructionForUsualHausdorffMeasures}.
\end{lemma}

\begin{proof}
Let $G(B)^{N}_{I_i;J_i},\, i=1,...,m$ be the generating disks (of generation $N$) that intersect $B$, i.e. such that $B \cap G(B)^{N}_{I_i;J_i} \neq \emptyset$. By assumption, 
\begin{equation}\label{mgeq2}
 m \geq 2.
\end{equation}
We denote the protecting disks associated to $G(B)^{N}_{I_i;J_i}$ by $P(B)^{N}_{I_i;J_i}$. Let also $P(B)^{N}_{\widetilde{I_j};\widetilde{J_j}},\, j=1,...,q$ be the protecting disks (if there are any) of generation $N$ that intersect $B$, but such that $B$ does not intersect $G(B)^N_{\widetilde{I_j};\widetilde{J_j}}$ (the corresponding generating disks.)\\
We can assume that 
$$|B \cap P(B)^{N}_{\widetilde{I_j};\widetilde{J_j}}| < \frac{1}{100} |B|$$ 
for all $j$.
Otherwise, 
the same proof as for Case 2b above yields the proof of \eqref{IntegralHolderCondition} for $B$.\\
We also know that $P(B)^{N}_{\widetilde{I_j};\widetilde{J_j}} \nsubseteq B$ for all $j$ (since $B \cap G(B)^{N}_{\widetilde{I_j};\widetilde{J_j}} = \emptyset$ for all $j$.) Hence, by Lemma \ref{IfDiskBDoesNotMeetMuchOfOtherDisksThenHalfBDoesNotMeetOtherDisks}, $D = \frac{1}{2}B$ satisfies 
\begin{equation}\label{DDoesnotMeetProtectingDisksSuchThat2DDoesNotMeetTheirGeneratingDisk}
D \cap P(B)^{N}_{\widetilde{I_j};\widetilde{J_j}} = \emptyset \text{  for all  } j.
\end{equation}
and of course 
$$\int_D J(z,\phi)\,dA(z)\simeq\int_BJ(z,\phi)\,dA(z).$$
Now we can repeat the scheme above (since the beginning of section \ref{CalculationHolderExponentPhi}) with $D$ instead of $B$. Thus, let us denote by $G(D)^{N}_{I_i;J_i},\,i=1,...,m'$ the generating disks of generation $N$ that intersect $D$, 
and let $P(D)^{N}_{I_i;J_i}$ be the associated protectors. If $m' \leq 1$, then we are reduced to the above Cases already dealt with. So we are left with the assumption $m'\geq 2$. 
\\
\\
In this case we have that for all $i=1,...,m'$
$$r(G(D)^N_{I_i;J_i})<r(D)<r(G^{N-1}_{I',J'}).$$
Indeed, $D\subset B\subseteq G^{N-1}_{I',J'}$ so that $r(D)<r(G^{N-1}_{I',J'})$. On the other hand, if $r(G(D)^N_{I_i;J_i}) \geq r(D)$, then $r(B) = 2r(D) \leq 2 r(G(D)^{N}_{I;J}) \ll r(P(D)^{N}_{I;J})$, and since $B \cap G(D)^{N}_{I;J} \neq \emptyset$, then $B \subset P(D)^{N}_{I;J}$, which contradicts equation \eqref{mgeq2}.\\
\\
Let us now explain the main advantage of working with $D$ instead of $B$. Let $P(D)^N_{\widetilde{I_j};\widetilde{J_j}}, j=1,...,q'$ be the protecting disks (if any such disk exists) of $N$-th generation that intersect $D$, and whose corresponding generating disks $G(D)^N_{\widetilde{I_j};\widetilde{J_j}}$ do not, i.e. $D\cap G(D)^N_{\widetilde{I_j};\widetilde{J_j}}=\emptyset$ for all $j$. We have that 
\begin{equation}\label{2DMeetsGeneratingDisksSuchThatDMeetsTheirProtectingDiskButNotTheGeneratingDisk}
2D \cap G(D)^{N}_{\widetilde{I_j};\widetilde{J_j}} \neq \emptyset \text{  for all  } j. 
\end{equation}
Otherwise, $P(D)^{N}_{\widetilde{I_j};\widetilde{J_j}}$ (which meets $D \subset B = 2D$) would be a protecting disk of the type $P(B)^{N}_{\widetilde{I_j};\widetilde{J_j}}$, contradicting \eqref{DDoesnotMeetProtectingDisksSuchThat2DDoesNotMeetTheirGeneratingDisk}. This is actually the key point for the end of the proof.\\
\\
We can now use Lemma \ref{IfDMeets2OrMoreGeneratingDisksThen4DContainsTheCorrespondingProtectingDisks} twice. On the one hand,
$$\bigcup_{i=1}^{m'} P(D)^{N}_{I_i;J_i} \subseteq 4D,$$ 
and on the other hand, due also to \eqref{2DMeetsGeneratingDisksSuchThatDMeetsTheirProtectingDiskButNotTheGeneratingDisk}, we have
$$\bigcup_{j=1}^{q'} P(D)^{N}_{\widetilde{I_j};\widetilde{J_j}} \subseteq 8D.$$
Notice that $\phi(P^N_{I_i,J_i})=\phi_{N-1}(P^N_{I_i,J_i})$ as sets (and analogously for $P^{N}_{\widetilde{I_j};\widetilde{J_j}}$), while $\phi=\phi_{N-1}$ out of the protecting disks of $N$-th generation. Hence, 
we get by Lemma \ref{ConsequenceJacobianDoublingMeasure} and Lemma \ref{CheckingIntegralHolderConditionForPhiNForDisksContainedAndConcentricToPN} (a),
$$
\aligned
\int_D J(z,\phi)\,dA(z)
&\leq\int_{D\cup\bigcup_{i=1}^{m'} P^{N}_{I_i;J_i}\cup\bigcup_{j=1}^{q'} P^{N}_{\widetilde{I_j};\widetilde{J_j}} }J(z,\phi)\,dA(z) = \\
&=\int_{D\cup\bigcup_{i=1}^{m'} P^{N}_{I_i;J_i}\cup\bigcup_{j=1}^{q'} P^{N}_{\widetilde{I_j};\widetilde{J_j}} }J(z,\phi_{N-1})\,dA(z) \leq \\
&\leq\int_{8D}J(z,\phi_{N-1})\,dA(z) \simeq 
\int_{D}J(z,\phi_{N-1})\,dA(z)\lesssim |D|^{\frac{t}{t'}}.
\endaligned
$$

\end{proof}

\section{
Proof of Theorem \ref{MainTheoremNonRemovableSetsAtCriticalDimension}
}\label{FinalConstructionOfHolderQRMap}

We write the following Lemma \ref{HolderExponentOfCauchyTransformOfMeasureWithGivenGrowth} for the reader's convenience, even though the arguments are known (see \cite{Ca}).

\begin{lemma}\label{HolderExponentOfCauchyTransformOfMeasureWithGivenGrowth}
Let $\alpha \in (0,1)$. Let $\mu$ be a positive Radon measure supported on a compact set $A \subset \C$, such that 
$$\mu(D(z,r))\leq C\,r^{1+\alpha}$$ 
for any $z\in A$. Its Cauchy transform $f={\cal C}\mu = \frac{1}{\pi} \; \mu\ast\frac{1}{z}$ defines a holomorphic function on $\C\setminus A$, not entire, and with a H\"older continuous extension to $\C$, 
with exponent $\alpha$.
\end{lemma}

As a consequence of Theorems \ref{TheoremConstructionForUsualHausdorffMeasures} and \ref{TheoremYieldingHolderExponentPhi} we can now prove our main result.

\begin{cor}\label{FinalConstructionKQRMappingForSetE}
Let $K\geq 1$ and $\alpha\in(0,1)$. For $d=2\frac{1+\alpha K}{1+K}$ there exists a compact set $E$ with~$0<\H^d(E)<\infty$, non removable for $K$-quasiregular mappings in $\Lip_\alpha (\C)$.
\end{cor}
\begin{proof}
If $K=1$, then the result follows by Dol\v zenko's work \cite{Do}. 
Let $E$ and $\phi$ be as in Theorems \ref{TheoremConstructionForUsualHausdorffMeasures} with $t=d$, so that $0<\H^t(E)<\infty$ and $0<\H^{d'}(\phi (E))<\infty$. By Frostman's Lemma, we can construct a positive Radon measure $\mu$ supported on $\phi(E)$, with growth~$d'$. By Lemma \ref{HolderExponentOfCauchyTransformOfMeasureWithGivenGrowth}, its Cauchy transform $g={\cal C}\mu$ defines a holomorphic function on $\C\setminus\phi(E)$, not entire, and with a H\"older continuous extension to the whole plane, with exponent $d'-1$. Set
$$f=g\circ\phi.$$
Clearly, $f$ is $K$-quasiregular on $\C\setminus E$ and has no $K$-quasiregular extension to $\C$. Indeed, if $\tilde{f}$ extends $f$ $K$-quasiregularly to $\C$, then $\tilde{g}=\tilde{f}\circ\phi^{-1}$ would provide an entire extension of $g$, which is impossible. Furthermore, by Theorem \ref{TheoremYieldingHolderExponentPhi}, $f$ is (locally) H\"older continuous with exponent $(d' -1)\frac{d}{d'}=\alpha$. This finishes the proof.
\end{proof}


{\it{Acknowledgements.}} Part of this work was done while the second author was visiting Universitat Aut\`onoma de Barcelona and Universidad Aut\'{o}noma de Madrid. We thank both institutions for their hospitality. The authors also wish to thank K. Astala, J. Mateu and J. Orobitg for nice conversations on the subject of the paper. The picture was taken from \cite{Ur}, where it was done with the help of Mar\'{i}a Jos\'{e} Mart\'{i}n, so we extend our thanks to her.

\vskip 1cm
\begin{itemize}

\item[]{Departament de Matem\`atiques, Facultat de Ci\`encies, Universitat Aut\`onoma de Barcelona, 08193-Bellaterra, Barcelona, Catalonia\\
Department of Mathematics and Statistics, P.O.Box 68 (Gustaf H\"{a}llstr\"{o}min katu 2b), FI - 00014 University of Helsinki, Finland\\
{\it E-mail address:} {albertcp@mat.uab.cat}\\
{\it URL adress:} {www.mat.uab.cat/$\sim$albertcp}
}

\item[]{Mathematics Department, 202 Mathematical Sciences Bldg., University of Missouri, Columbia, MO 65211-4100, USA\\ 
{\it E-mail address:} {ignacio@math.missouri.edu}}
\end{itemize}

\end{document}